\documentclass[11pt]{amsart}
\usepackage{amsmath,amssymb,amscd,cite}

\usepackage{latexsym}
\usepackage[dvips]{graphics}
\usepackage[latin1]{inputenc}
\usepackage{amsthm}
\usepackage{amsfonts}
\usepackage{comment}
\usepackage{asymptote}
\usepackage{graphicx,xy,pstricks}
\input xy
\xyoption{all}

\newcommand{\executeiffilenewer}[3]{%
 \ifnum\pdfstrcmp{\pdffilemoddate{#1}}%
 {\pdffilemoddate{#2}}>0%
 {\immediate\write18{#3}}\fi%
}

\newcounter{notes}

\newtheorem{theorem}{Theorem}[section]
\newtheorem{proposition}[theorem]{Proposition}
\newtheorem{lemma}[theorem]{Lemma}
\newtheorem{corollary}[theorem]{Corollary}
\newtheorem{observation}[theorem]{Observation}
\newtheorem*{theorem*}{Theorem}

\theoremstyle{remark}
\newtheorem{remark}{Remark}[section]

\theoremstyle{plain}      

\newcommand{\Z}{\mathbb{Z}}
\newcommand{\Q}{\mathbb{Q}}

\newcommand{\N}{\mathbb{N}}
\newcommand{\C}{\mathbb{C}}

\newcommand{\tr}{\operatorname{Tr}}

\newcommand{\SL}{\mathrm{SL}}
\newcommand{\GL}{\mathrm{GL}}

\newcommand{\petitsl}{\mathrm{sl}}
\newcommand{\id}{\mathrm{Id}}

\newcommand{\Aut}{\operatorname{Aut}}
\newcommand{\Out}{\operatorname{Out}}

\newcommand{\Hom}{\operatorname{Hom}}

\title[Paralelograms on groups]{The parallelogram identity on groups and deformations of the trivial character
in $\SL_2(\C)$}

\author{Julien March\'e}
\address{Sorbonne Universit\'e, Universit\'e Paris Diderot, CNRS,
Institut de Math\'ematiques de Jussieu-Paris Rive Gauche, IMJ-PRG,
F-75005, Paris, France}
\email{julien.marche@imj-prg.fr}
\author{Maxime Wolff}
\address{Sorbonne Universit\'e, Universit\'e Paris Diderot, CNRS,
Institut de Math\'ematiques de Jussieu-Paris Rive Gauche, IMJ-PRG,
F-75005, Paris, France}
\email{maxime.wolff@imj-prg.fr}
\begin{document}

\begin{abstract}

We describe on any finitely generated group $\Gamma$ the space of 
maps $\Gamma\to\C$ which satisfy the {\em parallelogram identity},
$f(xy)+f(xy^{-1})=2f(x)+2f(y)$.

It is known (but not well-known) that these functions correspond to
Zariski-tangent vectors at the trivial character of the character variety
of $\Gamma$ in $\SL_2(\C)$. We study the obstructions for deforming the
trivial character in the direction given by $f$. Along the way, we show
that the trivial character is a smooth point of the character
variety if $\dim H_1(\Gamma,\C)<2$ and not a smooth point if
$\dim H_1(\Gamma,\C)>2$.

\vspace{0.2cm}

\end{abstract}
\maketitle

\newsavebox{\auteurbm}
\newenvironment{Bonmot}[1]%
  {\small\slshape%
  \savebox{\auteurbm}{\upshape\sffamily#1}%
  \begin{flushright}}
  {\\[4pt]\usebox{\auteurbm}
  \end{flushright}\normalsize\upshape}
\begin{Bonmot}{Atahualpa Yupanqui - Cachilo Dormido}
  Se le ha'i volver chacarera.
\end{Bonmot}
\section{Introduction}

It is a classical undergraduate
exercise to show that any function
$f:\Z^n\to \C$ satisfying the {\em parallelogram identity} below is a quadratic form:
\begin{equation}\label{parallelogram}
f(xy)+f(xy^{-1})=2f(x)+2f(y).
\end{equation}
The identity, written multiplicatively for the purpose of generalisation,
holds for all $x,y\in \Z^n$.
Solving this equation for a general group is a nice and recreative question
which has already been studied although not completely solved as far as
we know (see \cite{FriisStetkaer} for instance).
The question sounds deeper once we relate it to the theory of character
varieties. This relation was first noticed by Chenevier in \cite{Chenevier}
in which this theory is partially developed.
Our interest in it grew out independently from
different motivations (skein theory and dynamics on character varieties).
Before explaining this relation, let us fix a finitely generated
group $\Gamma$ and give a complete description of the space
$\mathcal{P}(\Gamma)$ of all functions satisfying
Equation~\eqref{parallelogram}.


\subsection{Description of the parallelogram functions}\label{ssec:IntroDesc}

The first solutions are the {\em quadratic forms}, defined as
$f(\gamma)=b(\gamma,\gamma)$ with $b\colon\Gamma\times\Gamma\to\C$
a bimorphism, \textsl{i.e.}, a morphism in both variables.
They form a set we denote by
$\mathcal{Q}(\Gamma)$.
Interestingly, some groups admit other parallelogram functions,
contrary to the case of $\Z^n$.

The case $\Gamma=\langle a,b,c\rangle$ of a free group of rank $3$ gives the
simplest example. If $w$ is a reduced word in the variables 
$a^{\epsilon_1},b^{\epsilon_2},c^{\epsilon_3}$ 
where $\epsilon_1,\epsilon_2,\epsilon_3\in\{\pm 1\}$,
we count, with sign
$\epsilon_1\epsilon_2\epsilon_3$, all ways of extracting
$a^{\epsilon_1} b^{\epsilon_2} c^{\epsilon_3}$, up to cyclic permutation,
from the word $w$, and we substract, with sign $\delta_1\delta_2\delta_3$,
all ways of extracting $a^{\delta_1}c^{\delta_2}b^{\delta_3}$
inside $w$, up to cyclic permutation.
The resulting map, which satisfies for example $f(abc)=-f(cba)=1$, turns out
to be in $\mathcal{P}(\Gamma)$,
see Lemma~\ref{lem:Premiere}.
It is obviously not in $\mathcal{Q}(\Gamma)$,
as it does not factor through the abelianization of~$\Gamma$.

To understand this phenomenon in greater generality, let us introduce the
notion of polynomial functions on groups. We can linearise any map
$f:\Gamma\to \C$ and view it as a linear form on $\C[\Gamma]$, the
group algebra of $\Gamma$. Let us abuse notation and still denote it by $f$.
A map $f:\Gamma\to\C$ is polynomial of order $< n$ if
\[ f((\gamma_1-1)\cdots(\gamma_n-1))=0\quad
  \forall \,\gamma_1,\ldots,\gamma_n\in \Gamma, \]
see \textsl{e.g.} \cite[Chap.~5]{Passi}.
For instance a constant function has order 0, a morphism has order
$\leqslant 1$, a quadratic form
has order $\leqslant 2$, and we will see in Section~\ref{ssec:paralcubic}
that a parallelogram function $f$ has order $\leqslant 3$.
Moreover, for a parallelogram function, the map
\[ f\circ\varepsilon_3(\gamma_1,\gamma_2,\gamma_3)=
  f((\gamma_1-1)(\gamma_2-1)(\gamma_3-1)) \]
which may be thought of as a kind of differential of order 3 of $f$ 
is an alternating trimorphism on
$\Gamma\times\Gamma\times\Gamma$. 
We will see that a parallelogram function is quadratic if and only if
this third derivative vanishes hence we get an exact sequence
\begin{equation}\label{suiteexacteincomplete}
\xymatrix{0\ar[r]&\mathcal{Q}(\Gamma)\ar[r]& \mathcal{P}(\Gamma)\ar[r]^(.4)
{\varepsilon_3^*}& \Lambda^3 H_1(\Gamma,\C)^*.}
\end{equation}
This exact sequence was already noticed in \cite{Chenevier,FriisStetkaer}.
Our contribution so far is to describe completely the image of $\varepsilon_3^*$.
This requires some basic knowledge of group cohomology for which we refer
to~\cite{Brown}. The shortest way to formulate our result  is to consider
the dual of the cup-product map
$H^1(\Gamma,\C) \times H^1(\Gamma,\C)\to H^2(\Gamma,\C)$.
By universal coefficients, it may be seen as a map
$c:H_2(\Gamma,\C)\to \Lambda^2 H_1(\Gamma,\C)$,
which has a rather elementary description, as we will recall
in Section~\ref{ssec:suiteexacte}.
\begin{theorem}\label{theo:parallel}
  Given any finitely generated group $\Gamma$, the image of
  $\varepsilon_3^*$ in the sequence~\eqref{suiteexacteincomplete} is the
  space
  \[\mathcal{E}(\Gamma)=\left\lbrace
  \Phi\colon\Lambda^3H_1(\Gamma,\C)\to\C\mid\forall x\in H_1(\Gamma,\C),
  \forall y\in H_2(\Gamma,\C), \Phi(x\wedge c(y))=0
  \right\rbrace. \]
\end{theorem}
The proof is an elementary application of the Hopf formula for
$H_2(\Gamma,\Z)$ and our theorem gives a complete, and efficient
description of the space of parallelogram functions (see the comments
after the proof of Lemma~\ref{lem:id}).
As an application, we show that for a surface $\Sigma_g$ of genus $g\geqslant 2$,
the group
$\operatorname{Mod}(\Sigma_g)=\operatorname{Out}(\pi_1 \Sigma_g)$ acts by
precomposition on $\mathcal{P}(\pi_1\Sigma_g)$ in a way which recovers
the Johnson homomorphism on the Torelli group.


\subsection{Relation with the character variety}\label{introcarac}

We consider the space
\[ R(\Gamma) = \Hom(\Gamma,\SL_2(\C)) \]
of morphisms from $\Gamma$ to $\SL_2(\C)$. This is an affine algebraic set.
One can embed it --non canonically-- in $\C^{4n}$ by sending $\rho$ to the
tuple $(\rho(\gamma_1),\ldots,\rho(\gamma_n))$ where
$\gamma_1,\ldots,\gamma_n$ is a generating set of $\Gamma$.
The ring $A(\Gamma)$ of ``regular functions'' on $R(\Gamma)$ is the ring of
polynomials in the indeterminates $a_{i,j}^l$
($1\leqslant i,j\leqslant 2$, $1\leqslant l\leqslant n$)
where $a_{i,j}^l$ is the entry $(i,j)$ of $\rho(\gamma_l)$. Of course
these functions satisfy some relations: the relations among the generators
in $\Gamma$ and those telling that $\det \rho(\gamma_l)=1$.
It should be observed however that
the algebra $A(\Gamma)$ is not necessarily reduced.

The group $\SL_2(\C)$ acts algebraically on $R(\Gamma)$ by conjugation and
the character variety $X(\Gamma)$ is the algebraic quotient of $R(\Gamma)$ by
this action (see for instance \cite{Mukai}, Section 5.1).
As a topological space, $X(\Gamma)$ is the quotient of
$R(\Gamma)$ by the relation $\rho\sim\rho'$ if and only if
$f(\rho)=f(\rho')$ for all regular functions on $R(\Gamma)$ invariant
by conjugation.
Equivalently, $X(\Gamma)$ is the space of
{\em characters}, \textsl{i.e.},
maps from $\Gamma$ to $\C$ of the form $\chi_\rho(\gamma)=\tr \rho(\gamma)$
for some  $\rho\in \Hom(\Gamma,\SL_2(\C))$.
The {\em trivial} character is then simply that of the trivial representation,
mapping all  elements of $\Gamma$ to~$2$.

The subring $A(\Gamma)^{\SL_2(\C)}$ of invariant functions becomes by
definition the ring of regular functions on $X(\Gamma)$. Generators for
the ring of invariants by the group $\SL_2(\C)$ were known to
specialists of the late 19th century (see for instance~\cite{GraceYoung})
but the search for a complete description of the ring of invariants of
tuples of matrices of size $n$ started only with Artin in~\cite{Artin69}
and was completed by Procesi in~\cite{Procesi}.
The statement has been reformulated many times since then: we state here
the case of $\SL_2(\C)$ and postpone to the end of the introduction the
case of~$\GL_n(\C)$.
In this form, this statement first appears
in~\cite[Proposition 9.1]{BrumfielHilden}
and a simpler proof is due to Chenevier,
see~\cite[Proposition 2.3]{ChenevierHab}.
\begin{theorem}[see~\cite{BrumfielHilden,
  ChenevierHab}]\label{theo:procesi}
  The ring $A(\Gamma)^{\SL_2(\C)}$ is generated by the elements $t_\gamma$
  for $\gamma\in \Gamma$ where $t_\gamma(\rho)=\tr \rho(\gamma)$
  (and a finite number of them suffice).
  Moreover these functions satisfy $t_1=2$ and the famous trace identity
  \[ t_\gamma t_\delta=t_{\gamma\delta}+t_{\gamma\delta^{-1}},\quad
    \forall \gamma,\delta\in \Gamma \]
  and these relations generate the ideal of relations among them.
\end{theorem}
As a main consequence, any function $f\colon\Gamma\to \C$ satisfying the
relations $f(1)=2$ and
$f(\gamma)f(\delta)=f(\gamma\delta)+f(\gamma\delta^{-1})$ for all
$\gamma,\delta\in \Gamma$ has the form $f=\chi_\rho$ for some
representation $\rho\colon\Gamma\to \SL_2(\C)$.

We may think of the trivial character, defined by $t_\gamma=2$ for all
$\gamma$, as an origin for the character variety. Thus, it may be convenient
to write $t_\gamma=u_\gamma+2$. Now $A(\Gamma)^{\SL_2(\C)}$ is generated by
the functions $u_\gamma$, subject to the relations $u_1=0$ and
\begin{equation}\label{eq:u}
  u_{\gamma\delta}+u_{\gamma\delta^{-1}}=
  2u_\gamma+2u_\delta+u_\gamma u_\delta.
\end{equation}

In this note we are interested in the ``deformations'' of the trivial
character. This word ``deformations'' is rather ubiquitous and deserves to
be explicited. Actual deformations in the algebraic set $X(\Gamma)$
may be realized as analytic paths $t\mapsto f_t(\gamma)$, with
$f_t(\gamma)=tf_1(\gamma)+\frac{t^2}{2}f_2(\gamma)+\cdots$ for each
$\gamma$, which satisfy for any $\gamma,\delta\in \Gamma$ the following
equation:
\begin{equation}\label{eq:Ordren}
  f_n(\gamma\delta)+f_n(\gamma\delta^{-1})=2f_n(\gamma)+2f_n(\delta)+
  \sum_{k=1}^{n-1}\binom{n}{k}f_k(\gamma)f_{n-k}(\delta).
\end{equation}
From this perspective, understanding the deformations of the trivial
character consists in solving Equation~\eqref{eq:Ordren},
for $n$ gradually increasing. A solution for all $n$ then yields a
formal deformation of the trivial character and the Artin approximation
theorem (see~\cite{Artin}) tells that there exists a convergent series which
coincides with the formal one at any given order.

At a purely algebraic level, the trivial character corresponds to the
maximal ideal $m$ of $A(\Gamma)^{\SL_2(\C)}$ generated by all
$u_\gamma$, for $\gamma\in\Gamma$.
The {\em Zariski-tangent} space at the trivial character is then, by
definition, the dual to the vector space $m/m^2$.
In the vector space $m/m^2$, Equation~\eqref{eq:u} loses its term
$u_\gamma u_\delta$, and a Zariski-tangent vector at the trivial
character is then simply a map $\Gamma\to\C$ satisfying the parallelogram
identity.
In other words,
\begin{observation}
  Let $f_1\colon\Gamma\to\C$ be any function.
  Then, $f_1\in\mathcal{P}(\Gamma)$ if and only if the map
  $u_\gamma\mapsto tf_1(\gamma)$ defines an algebra morphism from
  $A(\Gamma)^{\SL_2(\C)}$ to $\C[t]/(t^2)$.
\end{observation}
Of course, this is nothing else than a solution of
Equation~\eqref{eq:Ordren} to the order $n=1$.
More generally, solutions of Equation~\eqref{eq:Ordren} to the order $n$
may be thought of as Zariski-jets
to the order $n$
at the trivial character.
Said more abstractly, they form the space of ring morphisms from
$A(\Gamma)^{\SL_2(\C)}$ to $\C[t]/(t^{n+1})$ that have zero constant term
on~$m$.

In this note we explore the problem of solving Equation~\eqref{eq:Ordren}
for small $n$. We find that a parallelogram function has two universal
obstructions: one at order~2 and one at order~3. Here is the precise result.
\begin{theorem}\label{theo:obstruction}
  Let $f_1\in \mathcal{P}(\Gamma)$ be a parallelogram function. 
  \begin{enumerate}
  \item If there exists an algebra morphism
    $f=tf_1+t^2f_2:A(\Gamma)^{\SL_2(\C)}\to \C[t]/(t^3)$
    then $f_1$ is a quadratic form.
  \item If there exists an algebra morphism
    $f=tf_1+t^2f_2+t^3f_3:A(\Gamma)^{\SL_2(\C)}\to \C[t]/(t^4)$
 then $f_1$ is a quadratic form of rank $\leqslant 2$.
 \end{enumerate}
\end{theorem}

We will see that there are no other ``universal obstructions'' in the sense
that if $\Gamma$ is a free group and $f_1\in \mathcal{P}(\Gamma)$ is a
quadratic form of rank $\leqslant 2$ then there is a complete deformation
$f=tf_1+t^2f_2+\cdots$. This fits into a more general result of independent
interest: any formal deformation of the trivial character (at all orders)
is the character of a formal deformation of a parabolic representation.
Let us state the precise result.
\begin{theorem}\label{theo:lifting}
  Let $f=tf_1+t^2 f_2+\cdots$
  and suppose that $u_\gamma\mapsto f(\gamma)$ defines an algebra morphism
  from $A(\Gamma)^{\SL_2(\C)}$ to $\C[[t]]$. Then there
  exists a representation $\rho\colon\Gamma\to \SL_2(\C[[t]])$ such that
  for all $\gamma\in \Gamma$,
  $$2+f(\gamma)=\tr \rho(\gamma).$$
  Notice that $\rho$ evaluated at $t=0$ has trivial character. Hence
  is a parabolic representation, that is, it takes values in the abelian
  group of unipotent upper triangular matrices.
\end{theorem}  
Using these results we show that the trivial character is a smooth
point of $X(\Gamma)$ if $\dim H_1(\Gamma,\C)<2$ and is not smooth
if $\dim H_1(\Gamma,\C)>2$. In the remaining case, $\dim H_1(\Gamma,\C)=2$, 
we will see that $X(\Gamma)$ is smooth at the trivial character, if and only
if $R(\Gamma)$ is smooth at the trivial representation; and we will deduce
some explicit criteria for this smoothness.

Finally, we obtain the following consequence of a
celebrated theorem of Stallings (see~\cite{Stallings}).
\begin{theorem}\label{theo:etale}
  Let $\phi\colon \Gamma_1\to \Gamma_2$ be a group homomorphism, that
  induces an isomorphism between $H_1(\Gamma_1,\Z)$ and $H_1(\Gamma_2,\Z)$
  and an epimorphism from $H_2(\Gamma_1,\Z)$ to $H_2(\Gamma_2,\Z)$. Then
  $\phi^*\colon X(\Gamma_2)\to X(\Gamma_1)$ is \'etale at the trivial
  character.
\end{theorem}
This latter property is an algebraic analogue of a local diffeomorphism. 
Concretely, this means that $\phi^*$ induces an isomorphism between the
spaces of Zariski-jets at the trivial character,
see \textsl{e.g.}~\cite[Proposition~3.26]{Liu}.

It seems interesting to extend the results of this article to more general
settings. We conclude this introduction by determining the functional
equation corresponding to the case of
$\GL_n(\C)$.
Let $A_n(\Gamma)$ be the algebra of regular functions on
$\Hom(\Gamma,\GL_n(\C))$, Procesi's theorem states that the functions
$t_\gamma(\rho)=\tr \rho(\gamma)$ still generate the invariant subalgebra
of $A_n(\Gamma)$ and gives a complicated list of relations.
We learned from~\cite[Chap.~2]{ChenevierHab} the following reinterpretation
in terms of pseudo-characters.
\begin{theorem}[Procesi~\cite{Procesi} reformulated by
  Chenevier~\cite{ChenevierHab}]
  Let $R$ be a $\C$-algebra and $T:\Gamma\to R$ be a central map
  (\textsl{i.e.}, invariant by conjugation) which maps $1$ to $n$.
  Then there exists a morphism of algebras
  $A_n(\Gamma)^{\GL_n(\C)}$ to $R$ mapping $t_\gamma$ to $T(\gamma)$
  if and only if the following Frobenius identity is satisfied:
  \[\forall \gamma_0,\ldots,\gamma_{n}\in \Gamma, \quad
  \sum_{\sigma\in \mathfrak{S}_{n+1}}\epsilon(\sigma)T^\sigma
  (\gamma_0,\ldots,\gamma_n)=0 \]
  where
  $T^{(i_1,\ldots,i_k)}(\gamma_0,\ldots,\gamma_n)=
  T(\gamma_{i_1}\cdots\gamma_{i_k})$
  and $T^{\sigma}=\prod_{j=1}^k T^{\sigma_j}$ if
  $\sigma=\sigma_1\cdots\sigma_k$ is the decomposition of $\sigma$ into
  cycles (including the trivial ones).
\end{theorem}
Such a map $T$ is called a {\em pseudo-character}. We can derive from this
theorem a higher rank analogue of parallelogram functions.
\begin{corollary}
  The Zariski-tangent space at the trivial character of
  $\Hom(\Gamma,\GL_n(\C))/\!/\GL_n(\C)$ is naturally isomorphic to the space
  of central functions $f\colon\Gamma\to\C$ satisfying the following equality:
  \begin{equation}\label{eq:Paraln}
    \forall \gamma_0,\ldots,\gamma_n\in \Gamma,
    \quad\sum_{k=1}^{n+1}\sum_{0\le i_1,\ldots,
    i_k\le n}\frac{(-1)^{k}}{k!}f(\gamma_{i_1}\cdots\gamma_{i_k})=0.
  \end{equation}
\end{corollary}
\begin{proof}
  Writing $T=n+\epsilon f$ in Frobenius identity with $\epsilon^2=0$ we get
  the formula
  \[ \sum_{\sigma\in \mathfrak{S}_{n+1}}\epsilon(\sigma)n^{c(\sigma)-1}
  \sum_{(i_1,\ldots,i_k)\text{ cycle of }\sigma}f(\gamma_{i_1}\ldots
  \gamma_{i_k})=0 \]
  where $c(\sigma)$ is the number of cycles of $\sigma$.
  Each time a cycle $(i_1,\ldots,i_k)$ appears in a permutation $\sigma$,
  this permutation induces a permutation $\sigma'$ of
  $\{0,\ldots,n\}\smallsetminus\{i_1,\ldots,i_k\}$ of $n+1-k$ letters,
  with $c(\sigma')=c(\sigma)-1$ and
  $\epsilon(\sigma')=(-1)^{k+1}\epsilon(\sigma)$.
  We can count the contribution of these terms by using the following
  formula due to Rouquier, see~\cite[Corollaire~3.2]{Rouquier}:
  \[ \sum_{\sigma\in \mathfrak{S}_l}\epsilon(\sigma)t^{c(\sigma)}=t(t-1)
  \cdots(t-l+1). \]
  We get that the coefficient of $f(\gamma_{i_1}\cdots\gamma_{i_k})$ is
  equal to $(-1)^{k+1}n(n-1)\cdots k$, but we have to divide by $k$ because
  the cycle $(i_1,\ldots,i_k)$ appears $k$ times in the first sum of the
  proof. Dividing by $-n!$ on both sides yields the result.
\end{proof}

\subsection{Further remarks}

{\bf 1.} We choose to state our results with coefficients in $\C$ for
simplicity. We can replace it \emph{mutatis mutandis} with any
algebraically closed field of characteristic~0. In fact, many statements are
true on $\Q$ or even $\Z[\frac{1}{2}]$, as the diligent reader may notice.

{\bf 2.} It is well-known (see \textsl{e.g.}~\cite{Weil}) that the
Zariski-tangent space of $R(\Gamma)$ at the trivial representation is the
space $Z^1(\Gamma,\operatorname{sl}_2(\C))$ which happens to be isomorphic
to $H^1(\Gamma,\operatorname{sl}_2(\C))$ and that all obstructions for
deforming the trivial representation live in
$H^2(\Gamma, \operatorname{sl}_2(\C))$.
In this perspective, experts in deformations should not be surprised by
the appearance of $H_2(\Gamma,\C)$ in Theorem~\ref{theo:parallel}.

{\bf 3.} The results of the present article apply to any central character.
Indeed, the group $H^1(\Gamma,\Z/2\Z)$ acts on the character variety by
mapping $t_\gamma$ to $\varepsilon(\gamma)t_\gamma$ where
$\varepsilon\in \Hom(\Gamma,\Z/2\Z)$. This action reduces the study of
central characters to the study of the trivial one.

{\bf Acknowledgements}
We are grateful to Louis Funar and Gw\'ena\"el Massuyeau for their interest,
their careful reading and their encouraging comments. We would also like to
thank Ga\"etan Chenevier for his interest and for pointing out to us the
beautiful theory of pseudo-characters.
Finally we thank the anonymous
referees for their useful comments and suggestions.


\section{Solving the parallelogram identity}

We may start to play with Equation~\eqref{parallelogram} and make the
following first observations.
\begin{lemma}\label{lem:Elem}
  Any function $f\in \mathcal{P}(\Gamma)$ satisfies the following
  identities for any $\gamma,\delta\in \Gamma$:
  \begin{enumerate}
  \item $f(\gamma^n)=n^2f(\gamma)$ for all $n\in \Z$,
  \item $f(\gamma\delta)=f(\delta\gamma)$,
  \item $f(\gamma\delta\gamma^{-1}\delta^{-1})=0$.
  \end{enumerate}
\end{lemma}
\begin{lemma}\label{lem:Premiere}
  The map $f\colon F_3=\langle a,b,c\rangle\to\Z\subset\C$
  of Section~\ref{ssec:IntroDesc} is in $\mathcal{P}(F_3)$.
\end{lemma}
We leave the proof of Lemma~\ref{lem:Elem} as an exercise, and now sketch
a proof of Lemma~\ref{lem:Premiere}.
\begin{proof}
  First note that a word $w$ in $a,a^{-1},b,b^{-1},c,c^{-1}$ does not need
  to be reduced for $f(w)$ to make sense, and inserting a letter and its inverse
  in $w$ does not change $f(w)$. Note also that for all $w$, we have
  $f(w^{-1})=f(w)$. Indeed, every pick of
  $a^{\varepsilon_1}b^{\varepsilon_2}c^{\varepsilon_3}$ in $w$, for instance,
  corresponds to a pick of $c^{-\varepsilon_3}b^{-\varepsilon_2}a^{-\varepsilon_1}$
  in $w^{-1}$: this reverses the sign but also reverses the cyclic order in
  which the letters appear, hence the contributions to $f(w)$ and $f(w^{-1})$
  are equal. Finally if $w_1,w_2$ are two words in $a,a^{-1},b,b^{-1},c,c^{-1}$
  then $f(w_1w_2)+f(w_1w_2^{-1})-2f(w_1)-2f(w_2)$ counts exactly the
  terms of type $a^{\varepsilon_1}b^{\varepsilon_2}c^{\varepsilon_3}$
  (or permutations) that use letters in both $w_1$ and $w_2$ (or $w_2^{-1}$),
  as the others cancel in the difference. Terms for which two letters are
  in $w_1$ cancel in the sum $f(w_1w_2)+f(w_1w_2^{-1})$ because of the change
  of sign of the power in the third letter, while those for which two letters
  are in $w_2$ cancel because the cyclic order in which the letters appear
  changes between the two terms, while the product of signs of the powers
  does not.
\end{proof}
Other identities as in Lemma~\ref{lem:Elem} can be obtained directly,
but in the next section we develop a more systematic approach which
will give us information not only on parallelogram functions but also on the
functions $f_n$ involved in their deformations.

\subsection{The map $p$ and powers of the augmentation ideal}

Let us start by analyzing Equation~\eqref{eq:Ordren}. Suppose it is solved
up to the order
$n-1$: solving it to the order
$n$ is the problem of finding
a map $f_n\colon\Gamma\to\C$, such that the map
\[ (\gamma,\delta)\mapsto f_n(\gamma\delta)+f_n(\gamma\delta^{-1})
-2f_n(\gamma)-2f_n(\delta) \]
is a prescribed function, in terms of a solution
$(f_1,\ldots,f_{n-1})$ of
Equation~\eqref{eq:Ordren} to the order $n-1$.
This suggests to study the operator
which sends a map $f\colon\Gamma\to\C$ to the map $\Gamma^2\to\C$ defined by
$(\gamma,\delta)\mapsto f(\gamma\delta)+f(\gamma\delta^{-1})-2f(\gamma)
-2f(\delta)$.
By linearizing all maps on $\Gamma$, we consider this operator as the
adjoint of
$p:\C[\Gamma]\otimes\C[\Gamma]\to\C[\Gamma]$ defined for all
$\gamma,\delta\in \Gamma$ by
\[ p(\gamma\otimes\delta)=\gamma\delta+\gamma\delta^{-1}-2\gamma-2\delta. \]
Depending on the context, this map can also be viewed as a map
$p:\Gamma\times \Gamma\to\C[\Gamma]$ and we will call it the parallelogram
map. More generally, we will often replace the symbol
$\otimes$ by a (less cumbersome) coma when we evaluate on basis elements,
maps defined on tensor products.
With this notation, Equation~\eqref{eq:Ordren} becomes:
\begin{equation}\label{eq:Oordren}
  f_n\circ p(\gamma\otimes\delta) =
  \sum_{k=1}^{n-1}f_k(\gamma)f_{n-k}(\delta).
\end{equation}

Let $\varepsilon_n\colon\C[\Gamma]^{\otimes n}\to\C[\Gamma]$ be the
linear map defined by
\[ \varepsilon_n(\gamma_1\otimes\cdots\otimes\gamma_n)
=(\gamma_1-1)\cdots(\gamma_n-1). \]
Recall for instance from \cite{Passi} that the augmentation ideal $I$ is the
kernel of the map $\C[\Gamma]\to\C$ sending every $\gamma\in \Gamma$ to $1$.
One sees that the range of the map $\varepsilon_n$ is the ideal $I^n$.
Dually, the elements of $(\C[\Gamma])^*$ vanishing on $I^{n+1}$ are the
polynomial maps of order $\leqslant n$.
These maps $\varepsilon$ combine well together, in the sense that for all
suitable $k$, $j$ and $n$ we have

\[
\varepsilon_n(\gamma_1,\cdots,\gamma_j,
\varepsilon_k(\gamma_{j+1},\cdots,\gamma_{j+k}),\cdots,\gamma_{n+k-1})=
\varepsilon_{n+k-1}(\gamma_1,\cdots,\gamma_{n+k-1}).
\]

Note also that if $f\colon\Gamma\to\C$ vanishes at $1$, then
$f\circ\varepsilon_2(a\otimes b)=f(ab)-f(a)-f(b)$ measures how far is
$f$ from being a morphism. More generally, we will repeatedly use the
following observation:
\begin{equation}\label{testmulti} 
f\circ\varepsilon_{n+1}=0\implies
(\gamma_1,\ldots,\gamma_n)\mapsto f((\gamma_1-1)\cdots(\gamma_n-1))
\text{ is a morphism in each variable.}
\end{equation}


\subsection{Parallelogram functions are cubic}\label{ssec:paralcubic}

We will denote by $\mathcal{C}(\Gamma)$ the set of maps $f:\Gamma\to \C$
satisfying $f(1)=0$, $f(\gamma)=f(\gamma^{-1})$ and
$f(\gamma\delta)=f(\delta\gamma)$ for all $\gamma,\delta\in \Gamma$.
As the generators $t_{\gamma}\in A(\Gamma)^{\SL_2(\C)}$ satisfy the same
relations, all functions $f_n$ involved in Equation~\eqref{eq:Ordren} are
elements of $\mathcal{C}(\Gamma)$; of course this also follows
from Equation~\eqref{eq:Ordren} by induction on~$n$
(the inductive step
follows by using Equation~\eqref{eq:Oordren} with the elements
$1\otimes 1$, $1\otimes\gamma$ and then
$\gamma\otimes\delta-\delta\otimes\gamma$).

The objective of this paragraph is to prove the following statement.
\begin{lemma}\label{lem:ParallCubiques}
  For every map $f\in\mathcal{C}(\Gamma)$ and every $a,b,c,d\in\Gamma$,
  we have
  \begin{equation}\label{eq:Formule4}
  2f\circ\varepsilon_4(a,b,c,d)= f\circ p\left(\begin{array}{c}
  \varepsilon_3(a,b,c)\otimes d+\varepsilon_3(b,c,d)\otimes a +
  \varepsilon_3(a,b,d)\otimes c + \varepsilon_3(c,a,d^{-1})\otimes b \\
  - \varepsilon_2(a,d^{-1})\otimes\varepsilon_2(b,c)-\varepsilon_2(b,d)
  \otimes\varepsilon_2(c,a)-\varepsilon_2(d,c^{-1})\otimes\varepsilon_2(a,b)
  \end{array}\right).
  \end{equation}
  Also, for all $f\in\mathcal{P}(\Gamma)$ and every $a,b,c\in\Gamma$, we
  have $f\circ\varepsilon_3(a,b,c)+f\circ \varepsilon_3(a,c,b)=0$ and
  \begin{equation}\label{eq:VersGGG}
  f(a[b,c])-f(a)= 2f\circ \varepsilon_3(a,b,c).
  \end{equation}
\end{lemma}
It follows that for all $f\in\mathcal{P}(\Gamma)$, the map
$f\circ\varepsilon_3$ is an alternate trimorphism on
$\Gamma\times\Gamma\times\Gamma$. Also, the right hand side of
Equation~\eqref{eq:Formule4} will be useful for studying higher order jets,
and Equation~\eqref{eq:VersGGG} will be used in the next paragraph.
\begin{proof}
  Observe that for all $a,b,c\in\Gamma$ we have
  \[ abc+acb=2ab+2ac+2bc-2a-2b-2c+
  p(ab\otimes c+ac\otimes b-a\otimes bc^{-1}-2b\otimes c). \]
  This yields for every map $f$:
  \begin{equation}\label{eq:3alter}
  f\circ \varepsilon_3(a,b,c)+f\circ \varepsilon_3(a,c,b)
  =f\circ p(\varepsilon_2(a,b)\otimes c+\varepsilon_2(a,c)\otimes b
  -\varepsilon_2(b,c^{-1})\otimes a).
  \end{equation}
  In particular if $f\in \mathcal{P}(\Gamma)$ then
  $f\circ\varepsilon_3(a,b,c)+f\circ \varepsilon_3(a,c,b)=0$.
  Now for every $f\in\mathcal{C}(\Gamma)$, the left hand side of
  Equation~\eqref{eq:3alter} is invariant under permutations of $(a,b,c)$.
  It follows that its right hand side has the same symmetries: for example,
  permuting $b$ and $c$ gives that for all
  $f\in\mathcal{C}(\Gamma)$, the map $f\circ p$ vanishes on
  $\varepsilon_2(b,c^{-1})\otimes a - \varepsilon_2(c,b^{-1})\otimes a$.
  We may obtain similarly other
  elements of $\C[\Gamma]\otimes\C[\Gamma]$ on which $f\circ p$ vanishes
  for any $f\in\mathcal{C}(\Gamma)$, including
  \begin{equation}\label{eq:NoyauP}
  \begin{array}{c}
  \varepsilon_2(a,b)\otimes c + \varepsilon_2(a, b^{-1})\otimes c
  - \varepsilon_2(c, a)\otimes b - \varepsilon_2(c^{-1}, a)\otimes b,
  \text{ or}\\
  \varepsilon_2(b, c)\otimes a-\varepsilon_2(c, b)\otimes a
  +\varepsilon_2(b, a)\otimes c -\varepsilon_2(a, b)\otimes c
  -\varepsilon_2(a, c^{-1})\otimes b + \varepsilon_2(c^{-1}, a)\otimes b.
  \end{array}
  \end{equation}
  Now, we apply Equation~(\ref{eq:3alter}) successively to $(a,bc,d)$,
  $(ca,d,b)$ and $(ab,d,c)$ to get:
  \begin{align*}
  f\circ \varepsilon_3(a\otimes bc\otimes d+a\otimes d\otimes bc) & =
  f\circ p(\varepsilon_2(bc,a)\otimes d+\varepsilon_2(bc,d)\otimes a
  -\varepsilon_2(a,d^{-1})\otimes bc), \\
  f\circ\varepsilon_3(ca\otimes d\otimes b+ca\otimes b\otimes d) & =
  f\circ p(\varepsilon_2(b,ca)\otimes d+\varepsilon_2(b,d)\otimes ca
  -\varepsilon_2(ca,d^{-1})\otimes b), \\
  f\circ\varepsilon_3(ab\otimes d\otimes c+ab\otimes c\otimes d) & =
  f\circ p(\varepsilon_2(ab, d)\otimes c+\varepsilon_2(ab,c)\otimes d
  -\varepsilon_2(d,c^{-1})\otimes ab).
  \end{align*}
  The alternating sum of these three identities leads to the equation we
  are after. The left part yields
  $2f\circ\varepsilon_4(a,b,c,d)
  +2f\circ\varepsilon_3(a,c,d)+2f\circ\varepsilon_3(a,d,c)$.
  Substracting the term
  $2f\circ\varepsilon_3(a,c,d)+2f\circ\varepsilon_3(a,d,c)$
  to the right-hand side, expanding it by linearity, and using the
  relations from~(\ref{eq:NoyauP}) we finally obtain our main formula,
  Equation~\eqref{eq:Formule4}.

  Let us turn to the proof of Equation~\eqref{eq:VersGGG}.
  For all $a,b,c\in\Gamma$, we have
  $a[b,c]=abc\cdot b^{-1}\cdot c^{-1}$, while $a=abc\cdot c^{-1}\cdot b^{-1}$.
  Hence we have
  $a[b,c]-a\equiv\varepsilon_3(abc,b^{-1},c^{-1})-
  \varepsilon_3(abc,c^{-1},b^{-1})$,
  where we denote by $\equiv$, in $\C[\Gamma]$, the equality modulo the
  subspace generated by all $\gamma\delta-\delta\gamma$ for
  $\gamma,\delta\in\Gamma$.
  Hence if $f\in \mathcal{P}(\Gamma)$ we conclude by using the multilinearity
  and antisymmetry of $f\circ \varepsilon_3$.
\end{proof}


\subsection{The parallelogram exact sequence}\label{ssec:suiteexacte}
In this section, we prove Theorem \ref{theo:parallel}.
By Lemma~\ref{lem:ParallCubiques},
if $f\in \mathcal{P}(\Gamma)$ then
$f\circ\varepsilon_3$ is an alternate trimorphism on
$\Gamma\times \Gamma\times \Gamma$. If $f\circ \varepsilon_3=0$ then setting
$b(\gamma,\delta)=\frac{1}{2}(f(\gamma\delta)-f(\gamma)-f(\delta))$
we observe that $b$ is a symmetric bimorphism, that is,
$f(\gamma)=b(\gamma,\gamma)$ is quadratic.
Summing up, we obtain the exact sequence~\eqref{suiteexacteincomplete} and
our remaining task is to determine the range of the map
$\varepsilon_3^*\colon
\mathcal{P}(\Gamma)\to\Lambda^3H_1(\Gamma,\C)^*$.

We will start by recalling the basics of homology that we need.
We may write $\Gamma$ as a quotient $F/R$, where $F$ is finitely generated
free group and $R$ is a normal subgroup.
For any group $\Gamma$, set $\Gamma^{(1)}=\Gamma$
and let $\Gamma^{(k+1)}=[\Gamma^{(k)},\Gamma]$ be the subgroup of 
$\Gamma$ generated by commutators
$[a,b]$ with $a\in \Gamma^{(k)}$ and $b\in \Gamma$.
The Hopf formula asserts then that
$H_2(\Gamma,\Z)=[F,F]\cap R/[F,R]$. If we further denote by
$R'$ (resp. $R''$) the subgroup generated by $R$ and $[F,F]$
(resp. $R$ and $[F,[F,F]]$),
we observe the inclusions $[F,R]\subset[F,R']\subset R''$, yielding the
following maps, where the first one mixes an inclusion and a quotient:
\[ [F,F]\cap R/[F,R]\to [F,F]/[F,R'] \to [F,F]/R''\cap[F,F]\to 1. \]
It is an easy exercise to check that the sequence above is exact, and
the Hopf formula applied first to $\Gamma=F/R$ and then to its
abelianization $F/R'$, enable to rewrite the sequence above as:
\begin{equation}\label{eq:Stallings}
  H_2(\Gamma,\Z)\to H_2(\Gamma/\Gamma^{(2)},\Z)\to \Gamma^{(2)}/\Gamma^{(3)}
  \to 1.
\end{equation}
This exact sequence is due to Hopf, and we learned it from~\cite{Stallings}.
Now, for the abelian group $\Gamma/\Gamma^{(2)}=H_1(\Gamma,\Z)$, we
have the classical identification,
$H_2(H_1(\Gamma,\Z),\Z)\simeq\Lambda^2 H_1(\Gamma,\Z)$, which
identifies commutators $[a,b]$ from the Hopf formula (in the relations
defining $H_1(\Gamma,\Z)$) with the corresponding wedges $a\wedge b$.
We denote by $c\colon H_2(\Gamma,\Z)\to\Lambda^2H_1(\Gamma,\Z)$ the
composition map, and still denote
$c\colon H_2(\Gamma,\C)\to\Lambda^2H_1(\Gamma,\C)$ after tensoring with $\C$.

Recall that we denoted by $\mathcal{E}(\Gamma)$ the space of linear maps
$\Phi:\Lambda^3 H_1(\Gamma,\Z)\to \C$ such that $\Phi(x\wedge c(y))=0$
for all $x\in H_1(\Gamma,\C)$ and $y\in H_2(\Gamma,\C)$.

\begin{lemma}\label{H2necessaire}
  If $f\in \mathcal{P}(\Gamma)$ then
  $f\circ\varepsilon_3\in \mathcal{E}(\Gamma)$.
\end{lemma}
\begin{proof}
  Let $x\in H_1(\Gamma,\Z)$ and $y\in H_2(\Gamma,\Z)$; it suffices to prove
  that $f\circ\varepsilon_3(x\wedge c(y))=0$. Let $a\in\Gamma$ be an element
  whose abelianization is $x$, and let $r\in[F,F]\cap R$ represent
  $y$; let us write $r=\prod_i[b_i,c_i]$. As $r$ maps to $1$ in $\Gamma$,
  we have $f(ar)-f(a)=0$, and by repeated use of Formula~\eqref{eq:VersGGG}
  this yields $2\sum_i f\circ\varepsilon_3(a,b_i,c_i)=0$, ie,
  $f\circ\varepsilon_3(x\wedge c(y))=0$.
\end{proof}

This proves the condition of Theorem~\ref{theo:parallel} and it remains to
prove that any $\Phi\in \mathcal{E}$ may be written $f\circ\varepsilon_3$
for some $f\in\mathcal{P}(\Gamma)$. We have two proofs for it, an explicit
and a more conceptual one. We present first the explicit formula, leaving
the tedious details to the reader and then move to the conceptual proof.

Let $\Phi\in\mathcal{E}(\Gamma)$. The abelianization of $\Gamma$ has the form
$H_1(\Gamma,\Z)=\Z^r\oplus \bigoplus_{i=1}^d \Z/p_i\Z.$ We choose generators
$a_1,\ldots,a_r$, $t_1,\ldots,t_d$ of $H_1(\Gamma,\Z)$ corresponding to the
above decomposition, lift them to $\Gamma$, and, abusively, still denote them
by the same letter. Every element $\gamma\in\Gamma$ can be written in the form
\begin{equation}\label{ecrireg}
  \gamma=a_1^{n_1}\cdots a_r^{n_r}t_1^{\alpha_1}\cdots
  t_d^{\alpha_d}\prod_{i=1}^q[h_i,k_i],
\end{equation}
with, for $i\in\{1,\ldots,d\}$, $\alpha_i\in\{0,\ldots,p_i-1\}$. We then put
\begin{equation}
  f(\gamma)= \sum_{i<j<k} n_in_jn_k\Phi(a_i\wedge a_j\wedge a_k)
  +2\sum_{i=1}^q\Phi(u\wedge h_i\wedge k_i),
\end{equation}
where $u=a_1^{n_1}\cdots a_r^{n_r}$, and $h_i,k_i$ still (abusively) denote
their images in $H_1(\Gamma,\C)$. As announced, we encourage the reader to
check that this formula is well-defined, satisfies the parallelogram
identity and the equation $f\circ\varepsilon_3=\Phi$.

Let us avoid these tedious verifications and move to a more conceptual proof.
First, observe that all elements of $\mathcal{P}(\Gamma)$ factor through
$\Gamma/\Gamma^{(3)}$, as a consequence of Equation~\eqref{eq:VersGGG}.
Thus, the parallelogram equation we have to solve is an equation on the
group $\Gamma/\Gamma^{(3)}$, which sits in the central extension
$0\to\Gamma^{(2)}/\Gamma^{(3)}\to\Gamma/\Gamma^{(3)}
\to\Gamma/\Gamma^{(2)}\to 0$.
Let us set $A=\Gamma^{(2)}/\Gamma^{(3)}\otimes\C$ and $B=H_1(\Gamma,\Z)$.
In order to solve the parallelogram equation, it is convenient to introduce
a different central extension $0\to A\to U\to B\to 0$ related to $\Gamma$,
defined as follows. The commutator map $\Gamma^2\to \Gamma$ induces an
antisymmetric bilinear map
$\colon H_1(\Gamma,\Z)^2\to \Gamma^{(2)}/\Gamma^{(3)}$
which, after tensoring with $\C$, gives an antisymmetric map
$\colon H_1(\Gamma,\C)^2\to A$; we denote (without distinction) these maps
by $\sigma$. With this notation, the Hopf exact
sequence~\eqref{eq:Stallings} becomes the exact sequence
\[ \xymatrix{ H_2(\Gamma,\C)\ar[r]^(0.45){c} &
  \Lambda^2H_1(\Gamma,\C) \ar[r]^(0.7){\sigma} & A\ar[r] & 0}. \]
The map $\sigma$ is a cocycle, that is, the set $U=A\times B$ endowed with
the product $(a,x)(b,y)=(a+b+\sigma(x,y),x+y)$ is a group which fits into
a central extension as above. The advantage of $U$ is that it comes with a
``canonical'' parallelogram function as follows.
\begin{lemma}
  Let $\Omega=\Lambda^3 H_1(\Gamma,\C)/\operatorname{Span}\lbrace
  c(u)\wedge v, u\in H_2(\Gamma,\C),v\in H_1(\Gamma,\C)\rbrace$
  and let $F\colon U\to \Omega$ be defined by $F(a,x)=\alpha\wedge x$,
  where $\alpha$ is any element of $\Lambda^2H_1(\Gamma,\C)$ such that
  $\sigma(\alpha)=a$.
  Then $F$ is well-defined, and satisfies the parallelogram identity.
\end{lemma}

\begin{proof}
  Since the Hopf sequence above is exact, different choices of $\alpha$
  differ by elements of the form $c(u)$ with $u\in H_2(\Gamma,\C)$;
  and by definition of $\Omega$ these do not impact the value of $F(a,x)$.
  Now, if $(a,x),(b,y)\in U$ we have $(b,y)^{-1}=(-b,-y)$, so
  \[F((a,x)(b,y))+F((a,x)(b,y)^{-1})=(\alpha+\beta+x\wedge y)\wedge(x+y) +
  (\alpha-\beta-x\wedge y)\wedge (x-y), \]
  where $\alpha,\beta$ are lifts of $a,b$ to $\Lambda^2H_1(\Gamma,\C)$.
  Expanding this expression and simplifying, we get
  $2\alpha\wedge x+2\beta\wedge y$ as expected.
\end{proof}

Let us relate $\Gamma$ with $U$. A map $\Theta\colon\Gamma\to U$,
$\gamma\mapsto(\theta(\gamma),\overline{\gamma})$ (where the overline stands
for the abelianization) is a morphism if and only if for all
$\gamma,\delta\in\Gamma$ we have
$\theta(\gamma\delta)-\theta(\gamma)-\theta(\delta)=
\sigma(\overline{\gamma},\overline{\delta})\in A$.
Such a map exists if and only if the class of $\sigma$ is zero in
$H^2(\Gamma,A)$, and we claim that this holds tautologically.
As $A$ is a divisible group, the universal coefficients theorem tells
us that the evaluation map $H^2(\Gamma,A)\to\Hom(H_2(\Gamma,\Z),A)$
is an isomorphism. Hence it is sufficient to show that $\sigma$ vanishes on
generators of $H_2(\Gamma,\Z)$. By the Hopf formula, elements of
$H_2(\Gamma,\Z)$ are expressions of the form $r=\prod_i[x_i,y_i]$ which
vanish in $\Gamma$. By definition, the value of $\sigma$ on $r$ is the class
of $r$ in $A$ which is trivial by definition of $r$.
This proves the existence of such a map $\theta$.

Alternatively, and more explicitly, we may define two set-theoretic sections
$s_1,s_2\colon \Gamma_0\to\Gamma$ of the projection
$p\colon \Gamma\to\Gamma_0$, where $\Gamma_0$ is the quotient of
$H_1(\Gamma,\Z)$ by its torsion, by letting
$s_1(a_1^{n_1}\cdots a_r^{n_r})=a_1^{n_1}\cdots a_r^{n_r}$
and $s_2(a_1^{n_1}\cdots a_r^{n_r})=a_r^{n_r}\cdots a_1^{n_1}$ with
the notation of Equation~\eqref{ecrireg}. Then for $i=1,2$,
the map $\theta_i\colon x\mapsto x s_i(p(x))^{-1}$ may be viewed as a
map from $\Gamma$ to $A$, and it turns out $\theta=\theta_1+\theta_2$
is a solution.

Up to linearizing the map $F\circ\Theta\colon\Gamma\to\Omega$ as we did for
all parallelogram maps, we may compose it with the map $\varepsilon_3$, and
the following observation shows that $F$ is a ``universal'' solution
to the parallelogram problem.
\begin{lemma}\label{lem:id}
  The composition $F\circ\Theta\circ\varepsilon_3\colon\Omega\to\Omega$
  is the identity.
\end{lemma}
This concludes the proof of Theorem~\ref{theo:parallel}, as for any map
$\Phi\in\mathcal{E}(\Gamma)$, it suffices to set $f=\Phi\circ F\circ\Theta$
to get a map $f\in\mathcal{P}(\Gamma)$ such that $f\circ\varepsilon_3=\Phi$.
\begin{proof}[Proof of Lemma~\ref{lem:id}]
  Formally, $\varepsilon_3$ is defined only on $\C[\Gamma]^{\otimes 3}$.
  But since $F\circ\Theta$ satisfies the parallelogram identity,
  all the properties of parallelogram maps established above apply
  to it. In particular, $F\circ\Theta\circ\varepsilon_3$ reduces to
  a map $\Omega\to\Omega$, and we have, for all $x,y,z\in\Gamma$,
  \[
  2F\circ\Theta\circ\varepsilon_3(x,y,z)=
  F\circ\Theta(x[y,z])-F\circ\Theta(x) =
  \left( \widehat{\theta}(x[y,z])-\widehat{\theta}(x)\right)\wedge\overline{x},
  \]
  where $\widehat{\theta}(\gamma)$ is any lift of $\theta(\gamma)$ to
  $\Lambda^2H_1(\Gamma,\Z)$.
  Now, it follows from the defining formula of $\theta$ that
  $\theta(x[y,z])=\theta(x)+2\sigma(y,z)$; it follows that
  $2F\circ\Theta\circ\varepsilon_3(x,y,z)=
  2y\wedge z\wedge x=2x\wedge y\wedge z$.
\end{proof}

Now that the proof of Theorem~\ref{theo:parallel} is complete, let us add
a few words to mention that this description of $\mathcal{E}(\Gamma)$, or
dually, of the space $\Omega$ above can be computed effectively given a
finite presentation of $\Gamma$. Suppose $\Gamma=F/R$ where $F$ is the free
group on the letters $a_1,\ldots,a_r$ and $R$ its normal subgroup generated
by the words $r_1,\ldots,r_k$. Up to simple operations on the $r_j$, we may
suppose that for some $\ell$ their images
$\overline{r_1},\ldots,\overline{r_\ell}$ in the abelianization of $F$ freely
generate an abelian subgroup, and that $\overline{r_j}=0$ for all $j>\ell$.
All elements of $R$, resp.~$R\cap[F,F]$, are equivalent, modulo $[R,F]$, to
products of the form $r_1^{n_1}\cdots r_k^{n_k}$,
resp.~$r_{\ell+1}^{n_{\ell+1}}\cdots r_k^{n_k}$.
In other words, the abelian group $H_2(\Gamma,\Z)$ is finitely generated by
the elements $r_{\ell+1},\ldots,r_k$, although in general it may be difficult
to know if these elements satisfy some relations in $H_2(\Gamma,\Z)$.
Nevertheless $\Omega$ is the (computable) quotient of
$\Lambda^3H_1(\Gamma,\C)$ by all elements of the form $c(r_j)\wedge a_m$
with $j\geqslant\ell+1$ and where $a_m$ are generators of $H_1(\Gamma,\C)$.


\subsection{Examples}
For every $n\in \N$, we have two opposite examples. The first one is $F_n$
for which $H_2(F_n,\Z)=0$ and hence
$\mathcal{E}(F_n)=\Lambda^3H_1(F_n,\C)^*$, it is generated by the maps
raised in the introduction, for each choice of three generators.
The other example is $\Z^n$ for which $H_2(\Z^n,\Z)=\Lambda^2\Z^n$ and
$c\colon H_2(\Z^n,\Z)\to \Lambda^2 \Z^n$ is an isomorphism.
Hence $\mathcal{E}(\Z^n)=0$.

More interesting examples lie in between the previous ones. Let us give
some detail on the case of the fundamental group of closed orientable
surfaces of genus $g\geqslant 2$, denoting by $\Sigma_g$ this surface and by
$\Gamma_g$ its fundamental group.
Then $H_2(\Gamma_g,\Z)=H_2(\Sigma_g,\Z)=\Z$ is generated by the fundamental
class $[\Sigma_g]$. Moreover, if we have the presentation
\[ \Gamma_g=\langle a_1,b_1,\ldots,a_g,b_g \mid
  [a_1,b_1]\cdots[a_g,b_g]=1\rangle \]
then $c([\Sigma_g])=\sum_{i=1}^g a_i\wedge b_i$. This gives the description
of $\mathcal{E}(\Gamma_g)$ as the space of linear forms
$\Phi :\Lambda^3 H_1(\Gamma_g,\Z)\to\C$ such that
$\sum_{i=1}^g \Phi(x\wedge a_i\wedge b_i)=0$
for all $x\in H_1(\Gamma_g,\Z)$. This is trivial when $g=2$, as, for
example, $a_1\wedge a_2\wedge b_2$ equals $-a_1\wedge a_1\wedge b_1$ in
the quotient $\Omega$ of $\Lambda^3H_1(\Gamma,\C)$ by $H_2\wedge H_1$, but
it is nontrival as soon as $g\geqslant 3$.

The group $\Aut(\Gamma_g)$ acts on $\Hom(\Gamma_g,\SL_2(\C))$ by precomposition.
The induced action on $X(\Gamma_g)$ factors through the group
$\operatorname{Out}(\Gamma_g)$ of outer automorphism also known as the
(extended) mapping class group. This action has been extensively studied as
it extends the action of the mapping class group of the Teichm\"uller space,
a connected component of the real part of $X(\Gamma_g)$.
Goldman popularized many questions around the dynamics of this action,
see \textsl{e.g.}~\cite{goldman06}.
In some cases, understanding the neighbourhood of the trivial representation
may be useful, as in~\cite{FunarMarche} where the
{\em Torelli group} (\textsl{i.e.}, the subgroup
$\mathcal{I}_g$ of $\operatorname{Out}(\Gamma_g)$ acting trivially on
the abelianization $H_1(\Gamma_g,\Z)$) is shown to
act ergodically on some component of the real part of $X(\Gamma_g)$.
Let us show that the tangent action of $\Out(\Gamma_g)$ at the trivial
character in $X(\Sigma_g)$ is related to the Johnson homomorphism.

The group $\Aut(\Gamma_g)$ also acts on $\mathcal{P}(\Gamma_g)$
by precomposition and because parallelogram functions are invariant by
conjugation this action also factors through the group
$\operatorname{Out}(\Gamma_g)$.
By restriction, $\mathcal{I}_g$ acts on the exact
sequence~\eqref{suiteexacteincomplete},
and its action is trivial
on the extreme terms.
This
defines a morphism
$q\colon\mathcal{I}_g\to\Hom(\mathcal{E}(\Gamma),\mathcal{Q}(\Gamma))$,
such that for all $\phi\in\mathcal{I}_g$,
\[ f\circ\phi=f+q(\phi)(f\circ\varepsilon_3). \]
Recall from~\cite[Chap.~6]{farbmargalit} that the Johnson morphism
$\tau:\mathcal{I}_g\to \Hom(H_1(\Gamma_g,\Z),\Gamma_g^{(2)}/\Gamma_g^{(3)})$
is defined by the formula $\tau(\phi)(\overline{x})=\phi(x)x^{-1}$ for
$x\in\Gamma_g$. Then, for any $x\in \Gamma_g$ we get
\[ f(\phi(x))=f(x\cdot \phi(x)x^{-1})=
  f(x)+2f\circ\varepsilon_3(\overline{x}\wedge y), \]
by formula~\eqref{eq:VersGGG}, where $y$ is a lift of
$\tau(\phi)(\overline{x})$ in $\Lambda^2 H_1(\Gamma_g,\Z)$
(as in the Hopf exact sequence~\eqref{eq:Stallings}). This yields the simple
formula, $q(\phi)(\Phi)(x)=2\Phi(x\wedge \tau(\phi)(x))$: the action of the
Torelli group $\mathcal{I}_g$ on parallelogram functions is similar to the
Johnson homomorphism. This is of course valid for any group, the case of
surface groups being more classical.

The mix of binary and ternary elements in the same space
$\mathcal{P}(\Gamma)$, and the constant interplay between them, for example
by this Johnson morphism, evokes to us a genre of latino-american folk music.


\section{Obstructions and smoothness}

The description of parallelogram functions being done, we know completely
the Zariski-tangent space of $X(\Gamma)$ at the trivial representation.
We now turn to the jets of higher order.


\subsection{Proof of Theorem~\ref{theo:etale}}
As noticed from Formula~\eqref{eq:Formule4}, whenever $(f_1,\ldots,f_n)$
is a solution of Equation~\eqref{eq:Ordren} up to order $n$, we have
$f_i\in \mathcal{C}(\Gamma)$ for all $i\in\N$ and $f_1\circ\varepsilon_4=0$.
Equation~\eqref{eq:Oordren} gives $f_2\circ p= f_1\otimes f_1$.
By evaluating $f_2\circ\varepsilon_4$ at
an element of $(I^4)^4$ in Formula~\eqref{eq:Formule4},
we deduce that $f_2\circ\varepsilon_{16}=0$, and,
by immediate induction, $f_n\circ\varepsilon_{4^n}=0$: all solutions to
any order to the higher parallelogram equation~\eqref{eq:Ordren} are
polynomial. In the upcoming subsections we will obtain better bounds for
their orders; for now we deduce Theorem~\ref{theo:etale}.

As observed in the end of the proof of Lemma~\ref{lem:ParallCubiques},
for all
$a,b,c\in\Gamma$ the term $a[b,c]-a$ is equivalent to an element of $I^3$,
modulo elements of the form $\gamma\delta-\delta\gamma$ in $\C[\Gamma]$.
More explicitely,
\[ a[b,c] - a \equiv
  (abc-1)\left((b^{-1}-1)(c^{-1}-1)-(c^{-1}-1)(b^{-1}-1)\right). \]
By induction this formula gives that for all $a_1,\ldots,a_n\in\Gamma$, we
have $a_1[a_2,[a_3,\cdots[a_{n-1},a_n]\cdot\cdot\cdot]]\equiv a_1$
modulo $I^n$. It follows that for any solution $(f_1,\ldots,f_n)$ of
Equation~\eqref{eq:Ordren}, and for all $\gamma\in\Gamma$, the value of
$f_k(\gamma)$, for $k\in\{1,\ldots,n\}$ depends only on the image of
$\gamma$ in $\Gamma/\Gamma^{(4^n-1)}$. Now if
$\phi\colon\Gamma_1\to\Gamma_2$ is a morphism inducing an isomorphism
between $\Gamma_1/\Gamma_1^{(n)}\to \Gamma_2/\Gamma_2^{(n)}$ for all $n>0$,
then it induces a bijection between the set of solutions of the functional
equation~\eqref{eq:Ordren} at any order. It follows that
$\phi^*:X(\Gamma_2)\to X(\Gamma_1)$ induces an isomorphism between the
spaces of Zariski-jets at the trivial character at any order. With this in
head, Theorem~\ref{theo:etale} follows from the following theorem of Stallings.
\begin{theorem}[Stallings, Theorem~3.4 in~\cite{Stallings}]
  Let $\phi\colon \Gamma_1\to \Gamma_2$ be a morphism, that induces an
  isomorphism between $H_1(\Gamma_1,\Z)$ and $H_1(\Gamma_2,\Z)$ and an
  epimorphism from $H_2(\Gamma_1,\Z)$ to $H_2(\Gamma_2,\Z)$.
  Then $\phi$ induces an isomorphism between
  $\Gamma_1/\Gamma_1^{(n)}$ and $\Gamma_2/\Gamma_2^{(n)}$ for all
  $n\geqslant 1$.
\end{theorem}


\subsection{First obstruction}
Now we turn to the proof of Theorem~\ref{theo:obstruction}. Suppose
$(f_1,f_2)$ is a solution of Equation~\eqref{eq:Ordren} up to order~$2$.
In other words, we suppose that $f_1\in\mathcal{P}(\Gamma)$ and
$f_2\circ p(a,b)=f_1(a)f_1(b)$ for all $a,b\in\Gamma$. We will prove in
this section that this forces $f_1$ to be a quadratic form.

Take $a_1,\ldots,a_6\in\Gamma$. We have, using linearity and
Formula~\eqref{eq:Formule4},
\begin{align*}
  2f_2\circ\varepsilon_6(a_1,& \ldots,a_6) =
  2f_2\circ\varepsilon_4(\varepsilon_3(a_1,a_2,a_3), a_4,a_5,a_6) \\
  & = f_2\circ p\left(\begin{array}{c}
  \varepsilon_5(a_1,a_2,a_3,a_4,a_5)\otimes a_6+\varepsilon_3(a_4,a_5,a_6)
  \otimes\varepsilon_3(a_1,a_2,a_3) \\
  +\varepsilon_5(a_1,a_2,a_3,a_4,a_6)\otimes a_5+
  \varepsilon_5(a_5,a_1,a_2,a_3,a_6^{-1})\otimes a_4 \\
  -\varepsilon_4(a_1,a_2,a_3,a_6^{-1})\otimes\varepsilon_2(a_4,a_5)
  -\varepsilon_2(a_4,a_6)\otimes\varepsilon_4(a_5,a_1,a_2,a_3) \\
  -\varepsilon_2(a_6,a_5^{-1})\otimes\varepsilon_4(a_1,a_2,a_3,a_4)
  \end{array}\right).
\end{align*}
Since $f_1$ vanishes on $I^4$, this simplifies to
\[ 2f_2\circ\varepsilon_6(a_1,\ldots,a_6)=f_1\circ\varepsilon_3(a_1,a_2,a_3)
  f_1\circ\varepsilon_3(a_4,a_5,a_6). \]
As $f_1\in \mathcal{P}(\Gamma)$, the right hand side of this expression
changes sign upon exchanging $a_5$ with $a_6$, while the left hand side
remains equal upon permuting cyclically the $a_i$ because $f_2$ is invariant
by conjugation. In other words, the permutation $(5~6)$ changes the sign of
this complex number, while the permutation $(2~3~4~5~6~1)$ leaves it
invariant. As the latter has signature $-1$, it follows that for all
$a_1,\ldots,a_6$ we have $f_2\circ\varepsilon_6(a_1,\ldots, a_6)=0$,
and that $f_1\circ\varepsilon_3=0$, ie, $f_1\in\mathcal{Q}(\Gamma)$.

Note also that from all the above relations, we get that the map
$\Gamma^2\to\C$ defined by $\langle a,b\rangle= f_1\circ\varepsilon_2(a,b)$
is bilinear and symmetric, and that, for all $a,b,c,d\in\Gamma$, by applying
again Formula~\eqref{eq:Formule4} we have
\begin{equation}
  \label{eq:carre} f_2\circ\varepsilon_4(a,b,c,d)=
  \langle a,b\rangle \langle c,d\rangle+\langle a,d\rangle\langle b,c\rangle
  -\langle a,c\rangle\langle b,d\rangle.
\end{equation}
This expression being linear in each variable, we get that
$f_2\circ\varepsilon_5$ vanishes, hence $f_2$ is a polynomial function of
order~$\leqslant 4$.


\subsection{Second obstruction}

Now let $(f_1,f_2,f_3)$ be a solution to the order~$3$. Then
$f_2\circ\varepsilon_4$ satisfies Equation~\eqref{eq:carre}. We may write
different formulas for $f_3\circ\varepsilon_6$, by using the equalities
\[ \varepsilon_6(a_1,\ldots,a_6)=\varepsilon_4(\varepsilon_3(a_1,a_2,a_3),
  a_4,a_5,a_6)=\varepsilon_4(\varepsilon_2(a_1,a_2),\varepsilon_2(a_3,a_4),
  a_5,a_6).  \]
The first equality, together with Formula~\eqref{eq:Formule4} and the facts
that $f_1\circ\varepsilon_3$ and $f_2\circ\varepsilon_5$ vanish, gives
\begin{align*}
  2f_3\circ\varepsilon_6(a_1,\ldots,a_6)=&f_1\circ\varepsilon_2(a_4,a_5)
  f_2\circ\varepsilon_4(a_1,a_2,a_3,a_6)-f_1\circ\varepsilon_2(a_4,a_6)
  f_2\circ\varepsilon_4(a_5,a_1,a_2,a_3)\\
  &+f_1\circ\varepsilon_2(a_5,a_6)f_2\circ\varepsilon_4(a_1,a_2,a_3,a_4),
\end{align*}
while the second gives
\begin{align*}
  2f_3\circ\varepsilon_6(a_1,\ldots,a_6)=&f_1\circ\varepsilon_2(a_1,a_2)
  f_2\circ\varepsilon_4(a_3,a_4,a_5,a_6)+f_1\circ\varepsilon_2(a_5,a_6)
  f_2\circ\varepsilon_4(a_1,a_2,a_3,a_4)\\
  &-f_1\circ\varepsilon_2(a_3,a_4)f_2\circ\varepsilon_4(a_1,a_2,a_6,a_5).
\end{align*}
Now by expanding both these expressions, by using Equation~\eqref{eq:carre}
and by writing the equality between the two, we get an equality that fits
in the following determinant:
\begin{align*}
  \left|\begin{array}{ccc}
  \langle a_1,a_3\rangle & \langle a_4,a_3\rangle & \langle a_2,a_3\rangle \\
  \langle a_1,a_5\rangle & \langle a_4,a_5\rangle & \langle a_2,a_5\rangle \\
  \langle a_1,a_6\rangle & \langle a_4,a_6\rangle & \langle a_2,a_6\rangle
  \end{array}\right|=0.
\end{align*}
It follows that the induced bilinear form $\langle\cdot,\cdot\rangle$ on
$H_1(\Gamma,\C)$ cannot have a family of three orthonormal vectors: its
rank cannot exceed~$2$. This is a non-trivial condition as soon as
$\dim H_1(\Gamma,\C)\geqslant 3$.


\subsection{No other universal obstructions}

First, let us observe that there are no other ``universal'' -- i.e. valid
for any group -- obstructions of higher order.
\begin{proposition}\label{prop:PasPlusQueTrois}
  Let $\Gamma=F_n$ be a free group. Let $\langle\cdot,\cdot\rangle$ be a
  symmetric bilinear product of rank $\leqslant 2$ on $H_1(F_n,\C)$.
  Then there exists a smooth deformation
  of the trivial character, $[\rho_t]$, such that
  \[ \tr(\rho_t(\gamma))=2+t\langle\gamma,\gamma\rangle+
  o(t)\text{, for all }\gamma\in\Gamma. \]
\end{proposition}

\begin{proof}
  Let $F_n=\langle a_1,\ldots,a_n\rangle$ be a free group of rank~$n$.
  Being of rank $\leqslant 2$, the quadratic form associated to
  $\langle\cdot,\cdot\rangle$
  can be written as the product of  two linear forms $\ell_1,\ell_2$ on
  $H_1(F_n,\C)$ 
  (observe that $\ell_1^2+\ell_2^2=(\ell_1+i\ell_2)(\ell_1-i\ell_2)$).

  Now we construct $\rho_t$ as follows. We simply put
  \[ \rho_t(a_i)=\left(\begin{array}{cc}1 & \ell_1(a_i) \\ 0 & 1 \end{array}
  \right)\left(\begin{array}{cc}1 & 0 \\ t\ell_2(a_i) & 1\end{array}\right).\]
  We check that $\tr\rho_t(a_ia_j)=2+t\langle a_i,a_j\rangle+o(t)$,
  for all $i,j$.
\end{proof}


\subsection{Lifting deformations}\label{sec:lifting}

The proof of Proposition~\ref{prop:PasPlusQueTrois} suggests that
deformations of characters can be lifted to deformations of representations;
this is the content of Theorem~\ref{theo:lifting}, that we will prove now.
Thus, let us consider a function $f\colon\Gamma\to\C[[t]]$ satisfying
$f(\gamma\delta)+f(\gamma\delta^{-1})=f(\gamma)f(\delta)$ for all
$\gamma,\delta\in\Gamma$. By Theorem~\ref{theo:procesi},
$f$ can be viewed as an algebra morphism
$\phi\colon A(\Gamma)^{\SL_2(\C)}\to\C[[t]]$, which maps the function
$t_\gamma$ to $f(\gamma)$. We want to prove that there exists a morphism
$\rho\colon\Gamma\to\SL_2(\C[[t]])$ such that $\tr(\rho(\gamma))=f(\gamma)$
for all $\gamma\in\Gamma$.

Let $K$ be the field of fractions of $\C[[t]]$, \textsl{i.e.},
the field of formal
Laurent series in $t$, and let $\overline{K}$ be its algebraic closure.
By invariant theory over $\overline{K}$, the map $R(\Gamma)\to X(\Gamma)$
is surjective (see \textsl{e.g.}~\cite[Theorem~5.9]{Mukai}).
In particular, there exists
an algebra morphism $\overline{\phi}\colon A(\Gamma)\to\overline{K}$
extending $\phi$. This defines a representation
$\overline{\rho}\colon\Gamma\to\SL_2(\overline{K})$ by sending the generator
$\gamma_l$ to the matrix with entries $\overline{\phi}(a^l_{i,j})$ in the
notation of Subsection~\ref{introcarac}. In particular, for all $\gamma$,
$\tr(\overline{\rho}(\gamma))=f(\gamma)$. We want to prove that
$\overline{\rho}$ can be conjugated to a representation in $\SL_2(\C[[t]])$.
In fact, it suffices to conjugate it into $\SL_2(K)$, as the following
observation shows.
\begin{lemma}\label{lem:TousEnsemble}
  Let $\Gamma$ be a finitely generated group. Let $K$ be the field of
  fractions of $\C[[t]]$. Let $\rho\colon\Gamma\to\SL_2(K)$ and suppose that
  for all $\gamma\in\Gamma$, $\tr(\rho(\gamma))\in\C[[t]]$. Then $\rho$ is
  conjugate, in $\SL_2(K)$, to a representation in $\SL_2(\C[[t]])$.
\end{lemma}
\begin{proof}
  Let $v$ be the valuation on $K$ defined by $v(t)=1$. Then $\Gamma$ acts on
  the Bass-Serre tree $T$ associated to $(K,v)$.
  See \textsl{e.g.}~\cite[Chap.~2]{Serre}.
  For any $\gamma\in\Gamma$, we have $\tr(\rho(\gamma))\in\C[[t]]$, and
  it follows that $\rho$ is conjugate to an element of $\SL_2(\C[[t]])$:
  hence, $\rho(\gamma)$ fixes a vertex of $T$.
  Thus, $\Gamma$ acts on $T$ by isometry, in such a way that every element
  of $\Gamma$ fixes a point. Therefore, $\Gamma$ has a global fixed point,
  see \textsl{e.g.}~\cite[Chap.~I.6.5, Corollaire~3]{Serre}.
  This means that $\rho$ is conjugate to a representation in the stabilizer
  of a point of $T$, \textsl{i.e.},
  conjugate to a representation in $\SL_2(\C[[t]])$.
\end{proof}
A similar statement is true in the more general case of $n\times n$-matrices
but with the additional assumption of absolute irreducibility,
see~\cite[Lemma~1.4.3]{CullerShalen}. 

It remains to conjugate our representation $\overline{\rho}$ into $\SL_2(K)$.
We suppose first that it is (absolutely) irreducible, and leave the (easier)
reducible case to the end of the proof. Experts would notice that the result
follows from the fact that there are no non-trivial quaternion algebras over
$K$ (see~\cite{MaclachlanReid}). We prefer to give a down-to-earth proof.
The irreducibility condition is catched by pairs of elements of $\Gamma$, by
two classical observations that we recall now.
\begin{lemma}[Corollary 1.2.2 in~\cite{CullerShalen}]\label{lem:Irred-CS}
  Let $k$ be an algebraically closed field of characteristic zero,
  let $\Gamma$ be any group, and let $\rho\colon\Gamma\to\SL_2(k)$ be a
  representation. Then $\rho$ is irreducible if and only if there exist
  $\alpha,\beta\in\Gamma$ such that $\tr \rho([\alpha,\beta])\neq 2$.
\end{lemma}
\begin{lemma}\label{lem:gram}
  Let $k$ be any field and let $A,B\in\SL_2(k)$. Then the determinant of the
  Gram matrix of the family $(\id,A,B,AB)\in M_2(k)^4$, with respect to the
  non-degenerate bilinear form $(M,N)\mapsto\tr(MN)$ is equal to
  $- \left(\tr[A,B]-2\right)^2$.
\end{lemma}
This classical identity which we learned from~\cite{Saito} may be checked
by a direct computation. Fix $\alpha,\beta\in \Gamma$ given by
Lemma~\ref{lem:Irred-CS}, and set $A=\rho(\alpha)$ and $B=\rho(\beta)$.
Now we seek to conjugate $A$ and $B$ to the respective forms
\[ \begin{pmatrix} 0 & -1 \\ 1 & \tr A \end{pmatrix}\quad
  \text{ and }\quad\begin{pmatrix} a & b \\ c & d\end{pmatrix} \]
with $a,b,c,d$ in $K$ as follows. For $A$ it is just a matter of finding a
vector $v$ which is not an eigenvector of $A$ and considering the basis
$(v,Av)$ of $\overline{K}^2$. The entries of $B$ in this basis then satisfy
the system:
\begin{equation}\tag{S}
  a+d=\tr B,\quad b-c+d\tr A=\tr AB,\quad ad-bc=1;
\end{equation}
it will follow from Lemma~\ref{ilyaunesolution} below that this system
actually has solutions in $K$. Reciprocally, given a solution of $(S)$,
a simple computation shows that the matrix $X=B-a\id-c A$ has rank~1
(indeed we have $\tr X^2=(\tr X)^2$, and $X\ne 0$ as $A$ and $B$ do not
commute). Then for any non-zero vector $v\in \ker X$, we may check that
$(v,Av)$ is a basis in which $A$ and $B$ have the desired form.

By Lemma~\ref{lem:gram}, the matrices $\id,A,B,AB$ generate
$M_2(\overline{K})$. Hence, for any element $\gamma\in\Gamma$,
$\overline{\rho}(\gamma)$ is a linear combination of $\id,A,B,AB$ whose
coefficients are \textsl{a priori}
in $\overline{K}$. The values of $f(\gamma)$,
$f(\gamma \alpha)$, $f(\gamma\beta)$ and $f(\gamma\alpha\beta)$ yield a
system of four equations that enable to retrieve these four coefficients,
as its determinant is the Gram determinant of Lemma~\ref{lem:gram}.
Hence, it follows from the Cramer formula that these coefficients are in
$K$. Hence $\overline{\rho}$ takes values in $\SL_2(K)$.

To conclude with the proof of Theorem~\ref{theo:lifting} in this case, we
need to check that the system $(S)$ above has solutions in $K$. This is
the content of next lemma where we have set $x=\tr A, y=\tr B$ and
$z=\tr AB$. Recall from trace formulas that in this notation
$\tr [A,B]=x^2+y^2+z^2-xyz-2\ne 2$.

\begin{lemma}\label{ilyaunesolution}
  Let $x,y,z$ be in $K$ such that $x^2+y^2+z^2-xyz\ne 4$. Then there exists
  a solution $(a,b,c,d)\in K^4$ to the system 
  \[ a+d=y,\quad b-c+dx=z,\quad ad-bc=1. \]
\end{lemma}
\begin{proof}
  Eliminating $d$ and $c$ from the first two equations we get
  $a^2+b^2-abx-ay+b(yx-z)+1=0$. We complete the square by setting
  $a'=a-bx/2-y/2$ to get 
  \[ a'^2+b^2(1-x^2/4)+b(xy/2-z)+1-y^2/4=0. \]
  If $x=\pm 2$, we can easy solve the equation in $b$ unless $z=\pm y$,
  but this is forbidden by our assumptions. Hence, we factorize $1-x^2/4$
  and complete the square in $b$ to get
  $a'^2+(1-x^2/4)b'^2=\frac{x^2+y^2+z^2-xyz-4}{4-x^2}$.
  We can conclude from the following nice exercise: in $K$, any equation of
  the form $ax^2+by^2=1$ for $a,b\in K\smallsetminus\{0\}$ has a solution
  (hint: any non-zero element of $K$ has the form $x^2$ or $tx^2$).
\end{proof}

Now suppose finally that $\overline{\rho}$ is reducible. This implies that
$\overline{\rho}$ has the same character than a diagonal representation in
$\SL_2(\overline{K})$, so we will suppose that $\overline{\rho}$ is diagonal.
So $\overline{\rho}$ factors through the abelianization of $\Gamma$.
If $f(\gamma)=\pm 2$ for all $\gamma\in\Gamma$ then we may as well take
$\overline{\rho}$ to be the corresponding representation in
$\lbrace\pm \id\rbrace$. Thus, let us assume there exists $\gamma_0$ such
that $f(\gamma_0)\neq \pm 2$. Again, we may consider a vector
$v\in\SL_2(\overline{K})$ which is not an eigenvector of
$\overline{\rho}(\gamma_0)$ and conjugate $\overline{\rho}$ into the basis
$(v,\overline{\rho}(\gamma_0)v)$, by some element $g\in\SL_2(\overline{K})$.
This yields
$g\overline{\rho}(\gamma_0)g^{-1}=
\left(\begin{array}{cc}0 & -1 \\ 1 & f(\gamma_0)\end{array}\right)$,
and, for all $\gamma\in\Gamma$,
$g\overline{\rho}(\gamma)g^{-1}=
\left(\begin{array}{cc}x & y \\ -y & x-yf(\gamma_0)\end{array}\right)$,
as $\overline{\rho}(\gamma)$ and $\overline{\rho}(\gamma_0)$ commute.
Now, the equations $f(\gamma)=\tr \overline{\rho}(\gamma)$ and 
$f(\gamma_0\gamma)=\tr(\overline{\rho}(\gamma_0)\overline{\rho}(\gamma))$
yield the system
\[ \left\lbrace\begin{array}{rcrcl}2x &-& f(\gamma_0)y &=& f(\gamma) \\
f(\gamma_0)x&+&(2-f(\gamma_0)^2)y &=& f(\gamma_0\gamma)\end{array}\right. \]
whose determinant equals $4-f(\gamma_0)^2$, which is nonzero by hypothesis.
Hence, again by the Cramer formula, $x$ and $y$ lie in $K$, in other
words, $g\overline{\rho} g^{-1}$ takes values in $\SL_2(K)$ once again.


\subsection{Smoothness}

Let us begin by recalling some basics of algebraic geometry.
The dimension of a (Zariski) open set $U\subset X(\Gamma)$ is the maximal
length of a chain of irreducible closed subsets
$Z_0\subsetneq Z_1\cdots \subsetneq Z_n \subset U$.
The dimension of $X(\Gamma)$ at the trivial character $\chi$ is by definition
\[ \dim_\chi X(\Gamma)=\inf \{\dim U, \chi\in U \subset X(\Gamma)\}. \]
It is known that $\dim_\chi X(\Gamma)\leqslant \dim T_\chi X(\Gamma)$, and
$X(\Gamma)$ is said to be {\em smooth} at $\chi$ if the equality holds.
The meaning of this smoothness condition is that there are no obstructions
to interpolate any Zariski-tangent vector by an actual deformation of the
character, as we recall now.

Let $m\subset A(\Gamma)^{\SL_2(\C)}$ be the maximal ideal corresponding to
$\chi$. The smoothness of $X(\Gamma)$ at $\chi$ is equivalent to the
regularity of the localisation of $A(\Gamma)^{\SL_2(\C)}$ at $m$ that we
denote by $R$, see~\cite[Chap.~4.2]{Liu}. This property implies that the
completion $\widehat{R}$ of $R$ with respect to the filtration by powers of
$m$ is an algebra of power series in $\dim m/m^2$ variables
(see~\cite[Proposition~2.27]{Liu}).

Suppose that $X(\Gamma)$ is smooth at $\chi$. Any tangent vector
$f_1\in T_\chi X(\Gamma)=\mathcal{P}(\Gamma)$ can be viewed as a map
$f_1:A(\Gamma)^{\SL_2(\C)}\to \C[t]/t^2$ mapping $u_\gamma$ to $tf_1$.
This map extends to the localisation $R$ and maps $m^2$ to $0$.
As $\widehat{R}$ is an algebra of power series, it is easy to extend the
map $f_1$ to a full series $f=\sum_{i\geqslant 1} t^i f_i$ as in the
following diagram:
\[ \xymatrix{R\ar[r]^{f_1}\ar@{^{(}->}[d]& \C[t]/t^2 \\
\widehat{R}\ar@{.>}[r]^f& \C[[t]].\ar[u]} \]
The existence of $f$ shows that there are no obstructions for any tangent
vector~$f_1$.
\begin{proposition}
  Let $\Gamma$ be a finitely generated group and set $n=\dim H_1(\Gamma,\C)$.
  If $n<2$ then $X(\Gamma)$ is smooth at the trivial character if $n>2$,
  it is not.
\end{proposition}
\begin{proof}
  If $n=0$ then we computed that $\mathcal{P}(\Gamma)=0$. This proves that
  the trivial character is an isolated (and smooth) point of $X(\Gamma)$.
  If $n=1$ then $\dim \mathcal{P}(\Gamma)=1$. Moreover, there is a surjection
  $\Gamma\to \Z$ which induces an injection $X(\Z)\to X(\Gamma)$.
  The variety $X(\Z)$ is isomorphic to $\C$ (map $[\rho]$ to $\tr \rho(1)$)
  hence is 1-dimensional and contains the trivial character. It follows that
  $\dim_\chi X(\Gamma)\geqslant 1$ and again $X(\Gamma)$ is smooth at~$\chi$.
  
  Suppose now that $n\geqslant 3$. Any non-degenerate quadratic form
  $q\in \mathcal{Q}(\Gamma)$ is a tangent vector at $\chi$. If $X(\Gamma)$
  is smooth at $\chi$, it cannot be obstructed, but it follows from
  Theorem~\ref{theo:obstruction} that $q$ must have rank $\leqslant 2$
  and we get a contradiction.
\end{proof}

The remaining case where $\dim H_1(\Gamma,\C)=2$ appears to be more subtle,
and we will prove that the trivial character is smooth provided that
$H_2(\Gamma,\C)=0$. Before doing so, let us observe that
Theorem~\ref{theo:etale} already proves this statement under slightly
stronger hypothesis.
\begin{lemma}
  Suppose $H_1(\Gamma,\Z)\simeq\Z^2$ and $H_2(\Gamma,\Z)=0$.
  Then $X(\Gamma)$ is smooth at the trivial character.
\end{lemma}
\begin{proof}
  We may choose a morphism $\phi\colon F_2\to\Gamma$ that induces an
  isomorphism of the abelianizations. Then $\phi$ satisfies the hypothesis
  of Theorem~\ref{theo:etale}. Now, it is classical that $X(F_2)\simeq\C^3$
  is smooth at the trivial character; it follows that all Zariski-tangent
  vectors to the trivial character in $X(\Gamma,\SL_2(\C))$ are unobstructed,
  and hence, this is a smooth point.
\end{proof}
For example, if $H_1(\Gamma,\Z)\simeq\Z^2$ and $\Gamma$ admits a finite
presentation with two more generators than relations (such a presentation
is said to be of deficiency two), then we may check that $H_2(\Gamma,\Z)=0$,
following the comments after Lemma~\ref{lem:id} above. 
This also follows from the Epstein inequality, which
states
that the minimal number of generators of $H_2(\Gamma,\Z)$ is less than the
rank of $H_1(\Gamma,\C)$ minus the deficiency of $\Gamma$, see~\cite{Epstein}.
This gives many examples of groups with smooth trival character, as an
application of Theorem~\ref{theo:etale}.

Now we will extend this result to homology with complex coefficients. To this
end let us start with the observation that smoothness can be read in the
representation variety.

\begin{lemma}\label{lem:Rsuffit}
  Let $\Gamma$ be a group such that $\dim H_1(\Gamma,\C)=2$.
  Then, $X(\Gamma)$ is smooth at the trivial character if and only if
  $R(\Gamma)$ is smooth at the trivial representation.
\end{lemma}
\begin{proof}
  As $\dim H_1(\Gamma,\C)=2$, we have
  $\Lambda^3 H_1(\Gamma,\C)=0$, hence $\mathcal{E}(\Gamma)=0$.
  It follows that  $\mathcal{P}(\Gamma)=Q(\Gamma)$ has dimension~3.
  Also, $H^1(\Gamma,\operatorname{sl}_2(\C))$ \textsl{i.e.}, the Zariski-tangent
  space of $R(\Gamma)$ at the trivial representation has dimension~6.
  Therefore, the statement of the lemma would follow if
  the following inequality holds if we specialise the representation $\rho$
  to the trivial representation.
  \begin{equation}\label{dimfibre}
    \dim_{[\rho]} X(\Gamma) = \dim_\rho R(\Gamma)-3.
  \end{equation}

  This equality does not hold in general, but it does hold for an
  irreducible representation $\rho$ as, restricted to this open set, the
  quotient map is a ``geometric quotient'', \textsl{i.e.},
  each fiber consists in a
  single orbit of maximal dimension. In that case, the quotient map is flat
  because it is a locally trivial $\mathrm{PSL}_2(\C)$-principal bundle
  (see~\cite[Proposition 0.9]{GIT}) and the equality~\eqref{dimfibre}
  follows from general properties of flat morphisms
  (see \textsl{e.g.}~\cite[Theorem 3.12]{Liu}).
  Observe that one can prove it directly
  by constructing local cross-sections of the quotient map in the spirit of
  Section~\ref{sec:lifting}.

  Observe also that the space of reducible characters of $\Gamma$ is
  isomorphic to $X(\Z^2)$ hence has dimension~2. Its preimage in $R(\Gamma)$
  has dimension at most~5. Suppose that $R(\Gamma)$ is smooth at the trivial
  representation. It follows that every neighbourhood of the trivial
  representation contains an irreducible representation, at which $R(\Gamma)$
  still has dimension~6, and for which the equality~\eqref{dimfibre} holds.
  As the dimension is upper semi-continuous,
  the local dimension of $X(\Gamma)$ at the trivial
  character is at least~3. The converse holds for the same reason.
\end{proof}

We deduce some concrete criteria for the smoothness of the trivial
character in this case.

\begin{proposition}\label{criteres}
  Let $\Gamma$ be a group such that $\dim H_1(\Gamma,\C)=2$. If one of the
  following conditions holds, then $X(\Gamma)$ is smooth at the trivial
  character.
  \begin{enumerate}
  \item $H^2(\Gamma,\C)=0$.
  \item $\Gamma$ admits a finite presentation with deficiency $2$.
  \item There exists a surjection $\Gamma\to F_2$.
  \end{enumerate}
\end{proposition}

Condition~(2) holds, for example, for the fundamental group of a
non-orientable surface of genus~3.
Condition~(3) holds for the fundamental group
$\Gamma=\pi_1(S^3\smallsetminus L)$ of a homology boundary link
$L\subset S^3$ with two components.

\begin{proof}
  By Lemma~\ref{lem:Rsuffit} above, it suffices to prove that the trivial
  representation is a smooth point of $R(\Gamma)$. The first case is a
  standard result of deformation theory for which we refer
  to~\cite{NijenhuisRichardson}. Let us give a rough idea:
  following~\cite{Weil}, a tangent vector to the trivial representation is
  a cocycle, in our case, a morphism $z_1\colon\Gamma\to \petitsl_2(\C)$.
  The space $R(\Gamma)$ is smooth at $1$ provided that any $z_1$ gives rise
  to a morphism
  $\rho=\exp(\sum_{n\geqslant 1} t^n z_n)\colon\Gamma\to \SL_2(\C[[t]])$.
  One can prove its existence recursively by constructing
  $z_n\colon\Gamma\to\petitsl_2(\C)$ from the data of $z_k$ for $k<n$.
  Indeed, the equation $\rho(\gamma\delta)=\rho(\gamma)\rho(\delta)$ at
  the order $n$ can be written
  \[ z_n(\gamma\delta)-z_n(\gamma)-z_n(\delta)=F_n(\gamma,\delta) \]
  where $F_n$ is a linear combination of iterated brackets of $z_k(\gamma)$
  and $z_k(\delta)$ for $k<n$. It may be checked that $F_n$ is a $2$-cocycle
  (see~\cite{NijenhuisRichardson}), hence this equation has a solution as it
  can be written $dz_n=F_n$. This proves our assumption.
  
  In the second case, the remark following Lemma \ref{lem:id} 
  shows that $H_2(\Gamma,\Z)=0$ hence
  $H_2(\Gamma,\C)=0$ and we are done. Let us mention however that in this case
  the smoothness of $R(\Gamma)$ at $1$ is much easier to prove directly from
  the implicit function theorem: a presentation
  $\Gamma=\langle a_1,\ldots,a_n\mid r_1,\ldots,r_{n-2}\rangle$
  gives an embedding of $R(\Gamma)$ into $\SL_2(\C)^n$ which
  is smooth at the trivial representation. The reason is that
  $(r_1,\ldots,r_{n-2})$ is a submersion at
  $(1,\ldots,1)$ as $r_1,\ldots,r_{n-2}$ are linearly independent 
  in the abelian group generated by $a_1,\ldots,a_n$.
  
  In the third case, the surjection gives an inclusion
  $R(F_2)\subset R(\Gamma)$. As $R(F_2)=\SL_2(\C)^2$ has dimension~6, the
  conclusion follows.
\end{proof}

Although Proposition~\ref{criteres} covers many cases, it is not a closed
statement. Let us observe that its second condition gives a concrete
strategy as one can always extract from a presentation of a group $\Gamma$
a presentation with deficiency two of a group $\Gamma'$ that surjects on
$\Gamma$; then $X(\Gamma)\subset X(\Gamma')$ is smooth at the trivial
character if and only if the extra relations in $\Gamma$ are superfluous in
a neighbourhood of the trivial character.

A simple example is the group
$\Gamma=\langle a,b,c,d\mid c^3,d^3,(cd)^3\rangle$, the free product of
$F_2$ with the triangular group $(3,3,3)$. Indeed, close to the identity
in $\SL_2(\C)$, the equation $w^3=1$ as the unique solution $w=1$.
Hence the last relation is (locally) superfluous. Let us conclude with the
following more sophisticated example where the extra relations are globally
superfluous.

\begin{remark}
  Set $\Gamma'=\langle a,b,c,d\mid c^4[a,b]^2, d^3[a,b]^3\rangle$ and set
  \[ \Gamma=\langle a,b,c,d\mid c^4[a,b]^2, d^3[a,b]^3,
  [[c,[a,b]],[d,[a,b]]]\rangle. \]
  Let $\varphi\colon\Gamma'\to\Gamma$ be the most obvious morphism,
  mapping $a,b,c,d$ to $a,b,c,d$ respectively.
  Then, the associated map $\varphi_*\colon X(\Gamma)\to X(\Gamma')$
  is an isomorphism. However, $\varphi$ is not an isomorphism.
\end{remark}

\begin{proof}
  This amounts to saying that any representation of $\Gamma'$ in $\SL_2(\C)$
  factors through $\Gamma$. The key property is that, whenever
  $A,B\in\SL_2(\C)$ are two roots of a common non-central element
  (\textsl{i.e.}, if $\exists C\neq \pm1$, and $n,m$ such that $A^n=B^m=C$),
  then $A$ and $B$ commute. Let $A,B,C,D\in\SL_2(\C)$ be the respective
  images of $a,b,c,d$ by any morphism $\Gamma'\to\SL_2(\C)$. If $[A,B]=\pm1$
  then $[[C,[A,B]],[D,[A,B]]]=1$, obviously.
  If $[A,B]\neq \pm1$, then $[A,B]^2$ and $[A,B]^3$ cannot be both equal
  to $\pm 1$. If, say, $[A,B]^2\neq \pm 1$, then
  $C^{-1}$ and $[A,B]$ are both roots of this nontrivial element, hence they
  commute, and we have again $[[C,[A,B]],[D,[A,B]]]=1$.
  
  We still have to check that $\varphi$ is not an isomorphism. For this, it
  suffices to construct a morphism $\Gamma'\to \mathfrak{S}_6$
  which does not kill
  $[[c,[a,b]],[d,[a,b]]]$. One such example $\psi$ is defined as follows:
  $\psi(a)=(1~2~3)$, $\psi(b)=(1~4)(2~5)(3~6)$, $\psi(c)=(1~6~3~5~2~4)$
  and $\psi(d)=(2~3~4)$.
\end{proof}

%
%

\bibliographystyle{plain}
\bibliography{chacarera}

\end{document}